\documentclass{elsarticle}

\usepackage{amsmath, amsthm}
\usepackage{amssymb}

\usepackage{tikz, tikz-cd}
\usetikzlibrary{matrix,arrows,decorations.pathmorphing}

\newcommand{\susp}{\Sigma}

\newcommand{\ox}{\otimes}

\newcommand{\po}{\arrow[dr, phantom, "\ulcorner" near start]}
\newcommand{\pb}{\arrow[dr, phantom, "\lrcorner" near end]}

\newcommand{\inv}{\ensuremath{^{-1}}}

\newcommand{\term}[1]{{\bf \boldmath #1}} %\boldmath
\newcommand{\cat}{\mathcal}
\newcommand{\bb}{\mathbb}
\newcommand{\bbS}{\mathbb{S}}
\newcommand{\Z}{{\bb Z}}
\newcommand{\cof}{\rightarrowtail}

\newcommand{\coloneq}{\mathrel{\mathop:}=}

\newcommand{\bdry}{\partial}

\newcommand{\Proj}{\operatorname{Proj}}
\newcommand{\Mod}{\operatorname{Mod}}

\newcommand{\cofiber}{\operatorname{cofiber}}

\newcommand{\holim}{\operatorname{holim}}

\newcommand{\heart}{\heartsuit}

\newcommand{\Ho}{\operatorname{Ho}}

\newcommand{\oo}{\infty}
\newcommand{\x}{\times}

\newcommand{\ie}{\textit{i.e.,~}}

\newtheorem{theorem}{Theorem}

\newtheorem{definition}[theorem]{Definition}

\newtheorem{lemma}[theorem]{Lemma}

\newtheorem{remark}[theorem]{Remark}

\newtheorem{proposition}[theorem]{Proposition}

\newtheorem{corollary}[theorem]{Corollary}

\newtheorem{claim}[theorem]{Claim}

\newtheorem{conjecture}[theorem]{Conjecture}
\newtheorem*{conjecture*}{Conjecture}

\newtheorem*{example*}{Example}

\usepackage{fullpage}

\usepackage[obeyDraft]{todonotes}

\newcommand{\K}{\mathbf{K}}

\newcommand{\hocofib}{\operatorname{hocofib}}
\newcommand{\bbP}{\mathbb{P}}
\newcommand{\texttop}{\textrm{top}}
\newcommand{\textbottom}{\textrm{bot}}
\newcommand{\tohofib}{\operatorname{tohofib}}
\newcommand{\hofib}{\operatorname{hofib}}
\newcommand{\id}{\operatorname{id}}
\newcommand\surj{\twoheadrightarrow}
\newcommand{\inj}{\rightarrowtail}
\newcommand{\colim}{\operatorname{colim}}
\newcommand{\isom}{\cong}
\newcommand{\Nil}{\operatorname{Nil}}

\newcommand{\sk}{\operatorname{sk}}
\newcommand{\comp}{\circ}
\newcommand{\twiddle}{\widetilde}

\AtBeginDocument{%
   \def\MR#1{}
}

\begin{document}

\begin{frontmatter}

\title{A Fundamental Theorem for the $K$-theory of connective
  $\bbS$-algebras}

\author[1]{Ernest E.~Fontes\corref{cor1}}
\ead{fontes.17@osu.edu}
  %\fnref{ernie}}

\author[1]{Crichton Ogle}\ead{ogle.1@osu.edu}
  %\fnref{crichton}}

\cortext[cor1]{Corresponding author}

%\fntext[ernie]{Ernie's footnote text.}

%\fntext[crichton]{Crichton's footnote text.}

\address[1]{Department of Mathematics, The Ohio State University, 231
  W 18$^\textrm{th}$ Ave, Columbus, Ohio, USA}

% \address[2]{Department of Mathematics, The Ohio State University,
% Columbus, Ohio, USA}

\date{\today}

\begin{abstract}
  Invoking the density argument of Dundas--Goodwillie--McCarthy, we
  extend the Fundamental Theorem of $K$-theory from the category of
  simplicial rings to the category of $\bbS$-algebras. As an
  intermediate step, we prove the Fundamental Theorem for simplicial
  rings appealing to recent results from the first author's
  thesis. This recovers as a special case the Fundamental Theorem for
  the $K$-theory of spaces appearing in
  H\"uttemann--Klein--Vogell--Waldhausen--Williams.
\end{abstract}


\end{frontmatter}
%\maketitle

%%%%%%%%%%%%%%%%%%%%%%%%%%%%%%%%%%%%%%%%%%%%%%%%%%%

\section{Introduction}
The Fundamental Theorem of $K$-Theory (first formulated by Bass in low
dimensions and later extended by Quillen to all dimensions \cite{dg})
yields an isomorphism
\begin{equation*}
  K_*(R[t,t^{-1}])\cong K_*(R)\oplus K_{*-1}(R)\oplus NK^+_*(R)\oplus NK^-_*(R)
\end{equation*}
where $R$ is a discrete ring, $K_*(-)$ denotes its Quillen $K$-groups,
and
\begin{equation*}
  NK^{\pm}_*(R)\cong NK_*(R) \coloneq \ker(K_*(R[t])\overset{t\mapsto 0}{\longrightarrow} K_*(R))
\end{equation*}
The groups here are possibly non-zero in negative degrees, given that
they are computed as the homotopy groups of a (potentially)
non-connective delooping of the Quillen $K$-theory space arising from
a spectral formulation of this result \cite{cw2}. The nil-groups
$NK_*(R)$ capture subtle ``tangential'' information about $R$, and are
remarkably difficult to compute.

If $A$ is a connective $\bbS$-algebra, as defined in \cite[Chap.2]{dgm}, then $A[t]$, $A[t^{-1}]$, and
$A[t,t^{-1}]$ admit $\bbS$-algebra structures induced by that on $A$
in a natural way, which are connective if $A$ is. In this paper we
extend the Fundamental Theorem of (Waldhausen) $K$-theory to this
class of algebras. Precisely, we show:
\begin{theorem}[Fundamental Theorem for connective $\bbS$-algebras]\label{fundthm}
  For any connective $\bbS$-algebra $A$, there is a map of spectra
  \begin{equation*}
    \K(A) \to \susp^{-1} \hocofib \left( \K(A[t]) \vee_{\K(A)} \K(A[t\inv]) \to \K(A[t,t\inv]) \right )
  \end{equation*}
  which is functorial in $A$, induces equivalences on $(-1)$-connected
  covers (\ie $\pi_*$-isomorphism for $*\geq 0$), and splits a copy of
  $\K(A)\langle 0 \rangle$ off of
  $\left(\susp\inv \K(A[t,t\inv])\right)\langle 0 \rangle$.
\end{theorem}
For $\bbS$-algebras of the form $Q(\Omega(X)_+)$ ($X$ a basepointed
connected space), the Fundamental Theorem was first established in
\cite{hkvww1}.

We begin in section 2 by first verifying the main theorem when $A$ is
a simplicial ring via a sequence of results some of whose proofs are
deferred to the final section. In section 3, we extend the main
theorem to the $K$-theory of arbitrary connective $\bbS$-algebras via
a density argument inspired by \cite{dgm}.  The technical proofs we
provide for the simplicial ring case require working within an abelian
category, which prevents a direct proof for connective
$\bbS$-algebras. Section 4 documents various corollaries to the main
result, along with some natural conjectures which we hope to
investigate further in future work. Section 5 contains the majority of
the technical proofs needed to complete the results stated in section
2.  \vskip.2in

Although the Fundamental Theorem is a classical result for discrete
rings, no detailed proof exists in the literature in the
non-commutative case. Our proof of the theorem for simplicial rings
follows the outline given by Weibel \cite[Chap.V, Th.8.2]{cw2} which,
in turn, is based on the original approach of Quillen, appearing in
\cite{dg}. Some technical elements of the proof follow Lueck and
Steimle's approach from \cite{lueck}.  The Resolution Theorem (Theorem
\ref{resolution_thm}) requires retooling of Waldhausen's Sphere
Theorem \cite[Theorem 1.7.2]{fw2} in order to produce the required
results for simplicial rings. We utilize the reformulation in the
first author's PhD thesis \cite{fontes} to provide this key step of
the proof.

We would like to thank the referee for their careful reading and
insightful critique of an earlier version of this paper.

%%%%%%%%%%%%%%%%%%%%%%%%%%%%%%%%%%%%%%%%%%%%%%%%%%%

\section{The Fundamental Theorem for simplicial
  rings}\label{simprings}
Let $R$ denote a simplicial ring that is associative and unital but
not necessarily commutative. Let $\Mod(R)$ denote the category of
compact simplicial left $R$-modules.

\begin{definition} $\Proj^{str}(R)\subset \Mod(R)$ is defined as the full
  subcategory of $\Mod(R)$ whose objects are finitely generated projective simplicial $R$-modules $M$. Then $\Proj(R)$ is the larger (full) subcategory with objects $A$ occurring as retracts of $M\simeq M'$ with $M'$ an object in $\Proj^{str}(R)$.
\end{definition}
 
\begin{lemma}\label{lem2approx}
  Suppose $\cat A$ is a full Waldhausen subcategory of $\cat B$ and
  $\cat A$ admits a cylinder functor satisfying Waldhausen's cylinder
  axiom. If every object in $\cat B$ is a retract of an object in
  $\cat A$ through weak equivalences, then the induced map
  $K(\cat A)\to K(\cat B)$ is a weak equivalence of $K$-theory spaces.
\end{lemma}
\begin{proof}
  We apply Waldhausen's Approximation Theorem \cite[Theorem
  1.6.7]{fw2} to the inclusion $\cat A \to \cat B$. We check the
  approximation property: for any map $f:A\to B$ in $\cat B$ with
  $A\in \cat A$, we can pick a retract replacement for $B$ in
  $\cat A$, say $i: B \to B'$ with retract $r: B'\to B$ where both $r$
  and $i$ are weak equivalences and $B'$ is in $\cat A$. Then
  $i\comp f:A\to B'$ is in $\cat A$, since the subcategory is full,
  and $r \comp (i\comp f) = f$ as desired. The approximation theorem
  now implies $K(\cat A) \simeq K(\cat B)$.
\end{proof}
Lemma \ref{lem2approx} implies that $\Proj^{str}(R) \to \Proj(R)$
induces an equivalence on $K$-theory. We will write $K(R)$ for the
Waldhausen $K$-theory space $K(\Proj^{str}(R))\simeq K(\Proj(R))$.
\vskip.2in

Let $\textbf{Nil}(R)$ denote the category of pairs $(M, \phi)$ with $M$
an object of $\Proj(R)$ and $\phi:M\to M$ a homotopy-nilpotent endomorphism of
$M$ (i.e., $\exists n\ge 0$ with $\phi^n\simeq 0$). Morphisms in $\textbf{Nil}(R)$ are required to commute with these
endomorphisms. $\textbf{Nil}(R)$ forms a Waldhausen category and the
forgetful functor $\textbf{Nil}(R) \to \Proj(R)$ functorially splits off the
$K$-theory of $R$ as a summand of $K(\textbf{Nil}(R))$ (the functorial splitting being induced by $M\mapsto (M,0)$). We write $\Nil(R)$ for this reduced summand of the $K$-theory,
\begin{equation*}
  K(\textbf{Nil}(R)) \simeq K(R) \x \Nil(R)
\end{equation*}
with $\Nil_n(R) := \pi_n(\Nil(R))$.
\vskip.2in
%% define NK

Let $NK(R)$ denote the homotopy fiber of the map $K(R[t]) \to K(R)$
where $t\mapsto 0$. Since $R[t]$ and $R[t\inv]$ are abstractly
isomorphic, $NK(R)$ is also the homotopy fiber of the map
$K(R[t\inv])\to K(R)$ which maps $t\inv \mapsto 0$. We will denote the
former by $NK^+(R)$ and the latter by $NK^{-}(R)$ when it is
convenient to distinguish them. Their (canonically isomorphic) homotopy groups will be denoted, respectively, by
$NK_*(R)$, $NK^+_*(R)$, or $NK^{-}_*(R)$.

\begin{theorem}[generalizing {\cite[Thm.~8.1]{cw2}}]\label{thm1}
  For every simplicial ring $R$, 
  $\Nil_n(R) \cong NK_{n+1}(R)$ for all $n \geq 0$.
\end{theorem}

%% define modules over the projective line on R
%% give wald structure!! (technical prop)

In order to prove Theorem \ref{thm1}, we introduce the category of
modules over the projective line on $R$.

\begin{definition} $\Mod(\bbP^1(R))$ is the category with objects $(M_+, M_-, \alpha)$, where
$M_+ \in \Mod(R[t])$, $M_-\in \Mod(R[t\inv])$, and $\alpha$ is a weak equivalence
\begin{equation*}
  \alpha: M_+ \ox_{R[t]} R[t,t\inv] \xrightarrow{\simeq} M_- \ox_{R[t\inv]} R[t,t\inv].
\end{equation*}
Moreover $\alpha$ is required to be a morphism of $R[t,t^{-1}]$-modules admitting a homotopy-inverse $\alpha^{-1}$ which is also an $R[t,t^{-1}]$-module map (the particular choice of such inverse is not part of the data).
Morphisms $f:(M_+,M_-,\alpha) \to (N_+,N_-,\beta)$ in
$\Mod(\bbP^1(R))$ are pairs of morphisms $f_+:M_+\to N_+$ and
$f_-:M_-\to N_-$ that are compatible with $\alpha,\beta$, in that 
$\beta \comp (f_+ \ox \id) = (f_-\ox\id) \comp \alpha$. 
\end{definition}

We will refer to objects of $\Mod(\bbP^1(R))$ as $\bbP^1(R)$-modules.  In analogy with the above, we have

\begin{definition} $\Proj^{str}({\bbP^1(R)})\subset \Mod({\bbP^1(R)})$ is the full subcategory on objects $(M_+,M_-,\alpha)$  where $M_+\in \Proj^{str}(R[t])$ and $M_-\in \Proj^{str}(R[t\inv])$. Likewise, define $\Proj({\bbP^1(R)})$ to be the full subcategory of those $\bbP^1(R)$-modules $(M_+,M_-,\alpha)$ occurring as retracts of an object in $\Mod(\bbP^1(R))$ that are weakly equivalent (in $\Mod(\bbP^1(R))$) to an element $(M'_+,M'_-,\alpha')\in obj(\Proj^{str}({\bbP^1(R)}))$.
\end{definition}

Both categories form Waldhausen categories, with the inclusion $\Proj^{str}({\bbP^1(R)})\hookrightarrow \Proj({\bbP^1(R)})$ inducing a weak equivalence of Waldhausen $K$-theory spaces (the proof of this result is relegated to section \ref{technicalproofs} as Proposition \ref{proj_line_waldhausen_str}). Let $K(\bbP^1(R))$ be the Waldhausen
$K$-theory space $K(\Proj({\bbP^1(R)}))$.
\vskip.2in

Projecting onto $M_+$ resp.~$M_-$ produces natural exact functors
\[
i:\Proj({\bbP^1(R)}) \to \Proj(R[t]),\qquad j:\Proj(\bbP^1(R)) \to \Proj(R[t\inv])
\]
We will denote the induced
maps on $K$-theory spaces by $i$ and $j$ as well, and the maps on
homotopy groups by $i_*$ resp.~$j_*$. Additionally, there are functors
$u_i:\Proj(R) \to \Proj(\bbP^1(R)), i\in\mathbb Z$ defined by
$u_i(M) = (M[t],M[t\inv], t^i)$, where $t^i$ is the isomorphism
\[
M[t]\ox_{R[t]} R[t,t\inv] \xrightarrow[\cong]{_-\otimes t^i} M[t\inv] \ox_{R[t\inv]} R[t,t\inv]
\]
given by multiplication by $t^i$. We note the following useful
observation about the $u_i$.

\begin{proposition}\label{j-and-u_i}
  $j$ equalizes the $u_i$: $j\comp u_i \simeq j\comp u_k$ for all
  $i,k$.
\end{proposition} % proof obvious from definition.

Simplicial $R$-modules can be tensored over simplicial
sets. Precisely, given a simplicial set $X$ one forms the bisimplicial
set $R \ox X$: passing to the diagonal yields a simplicial free
$R$-module which we also denote by $R\ox X$.
\vskip.1in

\begin{definition}
  A module in $\Mod(R)$ has \term{projective height $\leq n$} if it is
  weakly equivalent to a retract of $R\ox X$ for some $n$-skeletal simplicial
  set $X$, \ie a simplicial set where $\sk_n X \simeq X$. A module in
  $\Mod(R)$ has \term{projective height $\geq n$} if it is weakly equivalent
  to a retract of $R\ox X$ for a simplicial set $X$ with
  $\sk_{n-1}X \simeq *$.  A module $(M_+, M_-, \alpha)$ in
  $\Mod(\bbP^1(R))$ has \term{projective height $\leq n$}
  (respectively, \term{projective height $\geq n$}) if $M_-$ and $M_+$
  are both of projective height $\leq n$ (respectively $\geq n$).
\vskip.1in
  We will denote by $\cat H_n (R)$ (respectively, $\cat H_n (\bbP^1(R))$)
  the full subcategory of $\Mod(R)$ (resp.,
  $\Mod(\bbP^1(R))$) on modules with projective height $\leq n$.
\end{definition}
\vskip.1in

The categories $\cat H_n(R)$ and $\cat H_n(\bbP^1(R))$ form Waldhausen
categories by Proposition \ref{proj_heights_waldhausen_str}
below. Since $\Mod(R)$ consists of all \emph{compact} modules, we
adopt the notation that $H_\oo(R) = \Mod(R)$ (and
$H_\oo (\bbP^1(R)) = \Mod(\bbP^1(R))$) as $\Mod(R)$ is the direct
limit of the subcategories $\cat H_n (R).$ The natural functors $i$
and $j$ defined above, which project a $\bbP^1(R)$-module onto its
first or second terms, preserve projective height $\leq n$. Hence they
restrict to functors $i,j:\cat H_n (\bbP^1(R)) \to \cat H_n (R[t^{\pm 1}])$. Also
the maps $u_i$ preserve projective height, producing
$u_i:\cat H_n(R) \to \cat H_n(\bbP^1(R))$.
\vskip.2in

There is a natural map $\mathbf{Nil}(R)\to \Mod(R[t])$ which sends
$(N, \phi)$ to $N$ with $t$ acting via $\phi$. Since $\phi$ acts
nilpotently up to homotopy, $N$ is homotopy-$t^n$-torsion for some $n$ ($t^n\simeq 0$), hence there is also a
natural map $\mathbf{Nil}(R) \to \Mod(\bbP^1(R))$ sending $(N,\phi)$
to $(N,0,0)$ where $t$ acts via $\phi$.
\vskip.2in
By definition $\cat H_0(R)$ consists of projective $R$-modules. We observe the following.
\begin{proposition}\label{prop3}
  As Waldhausen categories, $\cat H_0(R)$ is isomorphic to $\Proj(R)$,
  the category of compact projective $R$-modules. Likewise,
  $\cat H_0(\bbP^1(R))$ is isomorphic to $\Proj(\bbP^1(R))$.
\end{proposition}

%%% resolution theorem

\begin{theorem}\label{resolution_thm}
  For any simplicial ring $R$, the inclusions $\cat H_0(R) \subseteq
  \cat H_n(R) \subseteq \cat H_\oo(R)$ and $\cat H_0(\bbP^1(R))
  \subseteq \cat H_n(\bbP^1(R)) \subseteq \cat
  H_\oo(\bbP^1(R))$ induce equivalences on algebraic $K$-theory,
  \begin{equation*}
    K(\cat H_0(R)) \simeq K(\cat H_n(R)) \simeq K(\cat H_\oo(R))
  \end{equation*}
  and
  \begin{equation*}
    K(\cat H_0(\bbP^1(R))) \simeq K(\cat H_n(\bbP^1(R)))
    \simeq K(\cat H_\oo(\bbP^1(R))).
  \end{equation*}
\end{theorem}

\begin{theorem}[$K$-theory of the projective line]\label{projline}
  The maps $u_0$ and $u_1$ induce an equivalence
  \begin{equation*}
    K(R) \x K(R) \xrightarrow{\simeq}{} K(\cat H_0( \bbP^1(R))) \simeq K(\bbP^1(R))
  \end{equation*}
  after a choice of sum on the loopspace (which is canonical after
  passing to homotopy groups).
\end{theorem}
The proofs of these two results are delayed to section \ref{technicalproofs}.
\vskip.2in

%%%%%%%%%%%

As above, an object $M$ of $\cat H_i(R[t])$ is said to be {\it homotopy-$t^n$-torsion} if $0\simeq t^n*_-:M\to M$. Let $\cat H_{i,T}(R[t])$ denote the full subcategory of
$\cat H_i(R[t])$ whose objects are homotopy-$t^n$-torsion for some $n \geq 1$.
Similarly let $\cat H_{i,+}(\bbP^1(R))$ denote the full subcategory of
$\cat H_i(\bbP^1(R))$ on objects of the form $(M_+, 0,0)$. Noting that
$M_+\ox_{R[t]}{R[t,t\inv]}\simeq 0$ if and only if $M_+$ is
homotopy-$t^n$-torsion for sufficiently large $n$, we have the following lemma.
\begin{lemma}\label{lemma2}% weibel ref?
  There is an equivalence of categories
  $\cat H_{i,T}(R[t]) \to \cat H_{i,+}(\bbP^1(R)))$ induced by the map
  $M\mapsto (M,0,0)$.
\end{lemma}
We will only need to use $\cat H_{1,T}(\bbP^1(R))$ to model
$\mathbf{Nil}(R)$.
\begin{lemma}\label{lemma5}
  There is a functor $\mathbf{Nil}(R) \to \cat H_{1,T}(R[t])$ taking a
  module $M$ with homotopy-nilpotent endomorphism $\phi$ to $M$ with $t$ acting
  via $\phi$ and this functor is an equivalence of categories.
\end{lemma}
\begin{proof}
  For a module $(N,\phi)$ in $\mathbf{Nil}(R)$, we consider $N$ as an
  $R[t]$-module with $t$ acting via $\phi$: $t\cdot a = \phi(a)$. We
  have the following short exact sequence of $R[t]$-modules
  \begin{equation*}
    \begin{tikzcd}
      N[t] \ar[r, "\varphi"] & N[t] \ar[r] & N
    \end{tikzcd}
  \end{equation*}
  where $N[t] \coloneq N\ox_R R[t]$ and $\varphi$ is the unique
  $R[t]$-module extension of the map $N\to N[t]$ that takes
  $a \mapsto at-\phi(a)$. Hence $\varphi$ maps $at^i$ to
  $at^{i+1}-\phi(a) t^i$. $N$ is then the coequalizer of $\varphi$ and
  $0$. This resolution displays $N$ as having projective height
  $\leq 1$ and thus defines a functor
  $\mathbf{Nil}(R) \to \cat H_{1,T}(R[t])$\footnote{We note: this
  construction relies on the fact that we are working with modules over
  simplicial rings rather than arbitrary $\bb S$-algebras.}.
\vskip.2in

Conversely, a   module $M$ in $\cat H_{1,T}(R[t])$ is an $R$-module equipped with
  a homotopy-nilpotent action of $t$. This constructs an inverse.
\end{proof}

\begin{lemma}\label{lemma-composite-nil}
  The composite
  \begin{equation*}
    \begin{tikzcd}[row sep=tiny]
      \Proj(R) \ar[r] & \mathbf{Nil}(R) \ar[r] & \cat H_1(\bbP^1(R)) \\
      P \ar[r, mapsto] & (P,0) \ar[r, mapsto] & (P,0,0)
    \end{tikzcd}
  \end{equation*}
  induces a map $K(R) \to K(\cat H_1 \bbP^1(R))$ which is equivalent
  to $u_0-u_1$ (after a choice of loopspace subtraction which becomes
  canonical on homotopy groups).
\end{lemma}
\begin{proof}
  There is a natural transformation $u_1 \to u_0$ given by the map
  $(t, 1)$, \ie multiplication by $t$ on the $R[t]$-module and the
  identity on the $R[t\inv]$-module, since this sends $u_1(P)$ to
  $u_0(P)$ compatibly with their isomorphisms. The cofiber of
  $u_1 \to u_0$ is the functor
  $P\mapsto (P[t] / tP[t], 0 , 0) = (P,0,0)$. By the additivity
  theorem, the cofiber sequence $u_1\cof u_0 \to (P\mapsto (P,0,0))$
  splits on $K$-theory so $u_0-u_1$ is equivalent to the composite
  $P\mapsto (P,0,0)$ as desired.
\end{proof}
We note that as $(u_0,u_1)$ is an equivalence on $K$-theory spaces
$K(R)\x K(R) \to K(\bbP^1(R))$, so is $(u_0,u_0-u_1)$.
\vskip.2in

\begin{proof}[Proof of Theorem \ref{thm1}]
  Let $w_j$ denote the morphisms in $\cat H_1(\bbP^1(R))$ that $j$
  takes to weak equivalences in $\cat H_1 (R[t\inv])$.  Waldhausen's
  Fibration Theorem \cite[Thm.~1.6.4]{fw2} implies that on
  $K$-theory we get a fiber sequence of spaces
  \begin{equation}\label{localization-seq}
    \begin{tikzcd}
      K(A) \ar[r] & K(\cat H_1(\bbP^1(R))) \ar[r, "{\ell_j}"] & K(B)
    \end{tikzcd}
  \end{equation}
  where $A$ is the category of $j$-acyclics in $\cat H_1(\bbP^1(R))$,
  $B$ is the category $(\cat H_1(\bbP^1(R)),w_j)$ of $w_j$-local
  objects, and $\ell_j$ is localization at $w_j$.  The $j$-acyclics
  are precisely modules $(N_+,N_-,\beta)$ with $N_- \simeq 0$ as
  $R[t\inv]$-modules. The category $A$ contains the full subcategory on objects of the
  form $(M_+, 0, \alpha)$ and Waldhausen's Approximation Theorem
  \cite[1.6.7]{fw2} proves that $A$ and this subcategory have
  equivalent $K$-theory spaces. The key point in the verification here is the observation
  that any map $f: (M_+,0,\alpha) \to (N_+,N_-,\beta)$ in $A$, where
  $f_-:0 \xrightarrow{\simeq} N_-$, can be factored as
  \begin{equation*}
    \begin{tikzcd}
      (M_+,0,\alpha) \ar[r,tail,"\tilde{f}"]& (\twiddle N_+, N_-, {\twiddle\beta}\ ) \ar[r,
      "\simeq"] & (N_+,N_-,\beta)
    \end{tikzcd}
  \end{equation*}
  where $\tilde{f}_+:M_+\hookrightarrow \widetilde{N}_+ = cyl(f_+)$ is the inclusion of $M_+$ into the mapping cylinder of $f_+$, $\tilde{f}_- = f_-$, and
\[
\tilde{\beta}:\widetilde{N}_+\underset{R[t]}{\otimes} R[t,t^{-1}]\xrightarrow[\simeq]{pr\otimes Id} N_+\underset{R[t]}{\otimes} R[t,t^{-1}]\xrightarrow[\simeq]{\beta} N_-\underset{R[t]}{\otimes} R[t,t^{-1}]
\]
where $pr$ denotes the strong deformation retract of $\widetilde{N}_+$ onto $N_+$. By Lemmas \ref{lemma2}
and \ref{lemma5},
\begin{equation}\label{KH1TequalsKNil}
  K(A) \simeq K(\cat H_{1,T}(R[t])) \simeq K(\textbf{Nil}(R)).
\end{equation}

\begin{claim}
  The inclusion $B\to \cat H_1(R[t\inv])$ induces an isomorphism on
  $K$-theory groups $K_*(B) \isom K_*(R[t\inv])$ for $* \geq
  1$. Furthermore, $K_0(B) \isom K_0(R)$.
\end{claim}
\begin{proof}
  By construction, $B$ is equivalent (as Waldhausen categories) to its
  essential image $J$ in $\cat H_1(R[t\inv])$. Compose with the
  inclusion
  $\cat H_1(R[t\inv]) \to H_\oo (R[t\inv]) \simeq \Proj(R[t\inv])$ to
  consider $J$ as a subcategory of $\Proj(R[t\inv])$. By Theorem
  \ref{resolution_thm}, the inclusion induces an equivalence on
  $K$-theory $K(\cat H_1(R[t\inv]))\simeq K(\Proj(R[t\inv]))$ and from
  Lemma \ref{lem2approx} we know $K(\Proj(R[t\inv]))$ is equivalent to
  the $K$-theory of the subcategory of strict projectives
  $\Proj^{str}(R[t\inv])$. We claim $J$ is cofinal in
  $\Proj(R[t\inv])$.
  
  The full subcategory of free $R[t\inv]$-modules and their
  isomorphisms is contained in $J$. Let $u$ be the $R$-module
  isomorphism $R[t\inv]\isom R[t]$ that sends $t\inv$ to $t$. A free
  $R[t\inv]$-module $M$ and an isomorphism $T:M\isom M$ can then be
  lifted to the module $(uM, M, u\inv \ox \id)$ and the endomorphism
  $(uT, T)$ in $B$. Since free $R[t\inv]$-modules are cofinal in
  $\Proj(R[t\inv])$ the Cofinality Theorem \cite{gersten} implies that
  $K_n(B)\isom K_n(R[t\inv])$ for $n\geq 1$.
  
  It remains to determine $K_0$. To that end, consider the diagram
  \begin{equation*}
    \begin{tikzcd}
      \cat H_1(R) \ar[d, "u_0"] \ar[r, "{-\ox_R R[t\inv]}"] & \cat
      H_1(R[t\inv]) \\ \cat H_1(\bbP^1(R)) \ar[r, "\ell_j"] & B \ar[u,
      "\twiddle j"]
    \end{tikzcd}
  \end{equation*}
  which relates basechange to $u_0$ and localization at $w_j$. Since
  the above diagram commutes, $K_0(B)$ must contain a summand of
  $K_0(R)$. But as $j$ equalizes the $u_i$ by Proposition
  \ref{j-and-u_i}, $K_0(\bbP^1(R))\isom K_0(R)\oplus K_0(R)$ by
  Theorem \ref{projline}, and as localization induces a surjection
  $(\ell_j)_0:K_0(\bbP^1(R)) \twoheadrightarrow K_0(B)$, $K_0(B)$ can
  contain at most one copy of $K_0(R)$. Hence, $K_0(B)\isom K_0(R)$.
\end{proof}  
%\vskip.2in
  
  Theorem \ref{resolution_thm} implies
  $K(\cat H_1(\bbP^1(R))) \simeq K(\cat H_0(\bbP^1(R)))$; then
  Proposition \ref{prop3} and Theorem \ref{projline} produce an equivalence
  $K(\cat H_0(\bbP^1(R))) \simeq K(R)\x K(R)$. The fibration theorem
  now yields the following long exact sequences when $n \geq 1$
  \begin{equation}
    \begin{tikzcd}      
      \cdots \ar[r] &K_n(R) \x \Nil_n(R) \ar[d, "{\cong}"] \ar[r,"f_*"] & K_n(R)\x K_n(R) \ar[d,
      "{((u_0)_*, (u_0-u_1)_*)}", "\cong"'] \ar[r,"g_*"] & K_n(R)\x NK_n^-(R)
      \ar[d, "{\cong}"] \ar[r] & \cdots\\
      \cdots \ar[r]& K_n(\textbf{Nil}(R)) \ar[r] & K_n(\cat H_0(\bbP^1(R)) \ar[r,
      "{j_*}"] & K_n(B) \isom K_n(R[t\inv]) \ar[r] & \cdots
    \end{tikzcd}
  \end{equation}
  where the middle vertical map is the map on homotopy groups induced by the map of spaces
  \[
  (u_0,(u_0 - u_1)): K(R)\times K(R)\to K(\cat H_0(\bbP^1(R))
  \]
 given as the loopspace sum of $u_0$ on
  the first component and $u_0-u_1$ on the second. The long exact
  sequence continues to $n=0$ but terminates at 
  $K_0(B)$ and our analysis above shows that $j_*$ surjects
  onto $K_0(B)\cong K_0(R)$.  Our goal will be to show that the restriction of $f_*$ to $\Nil_*(R)$ is zero for$*\ge 0$.
  \vskip.2in
  
 On the top row $f_*$ is induced by the map
 \[
f:  K(R)\x Nil(R)\simeq K({\bf Nil}(R))\simeq K(A)\to K(\cat H_1(\bbP^1(R))\simeq K(\cat H_0(\bbP^1(R)))\xrightarrow{h} K(R)\x K(R)
 \]
 where $h$ denotes a choice of homotopy inverse to the weak equivalence $(u_0,u_0 - u_1)$; this results in a weakly commuting diagram of spaces which becomes strictly commuting upon passage to homotopy groups. To the right, $g_*$ is induced by the map of spaces $g$ which takes the first component of $K(R)\times K(R)$ to the corresponding component of $K(R)\x NK^-(R)$ by the identity, and maps the second component to the basepoint. On the first component of $K_n(R)\x K_n(R)$, the right square commutes since $j \comp u_0$ takes an $R$-module $P$ to $P\ox_R R[t\inv]$ so the induced map is the basechange map. Since $j\comp u_0\simeq j\comp u_1$ by Proposition \ref{j-and-u_i}, the right square also commutes on the right factor of $K_n(R)\x
  K_n(R)$. Therefore the right square also commutes on homotopy groups and weakly commutes on the level of spaces.
\vskip.2in

  By construction, the map $K(\cat H_0(\bbP^1(R))) \to K(R[t\inv])$
  induces a surjection onto the $K(R)$ summand on homotopy groups. The
  left map in the top sequence maps $K_n(R)$ to the second factor of
  $K_n(R)\x K_n(R)$.  Our analysis shows that the long exact sequence
  on $K$-groups decomposes as the direct sum of the split-short-exact sequences
  \begin{equation*}
    \begin{tikzcd}
      0 \ar[r] &K_n(R) \ar[r, "\iota_2"] &
      K_n(R)\oplus K_n(R) \ar[r, "p_1"]& K_n(R) \ar[r] & 0
    \end{tikzcd}
  \end{equation*}
  (where $\iota_2$ denotes inclusion onto the second factor and $p_1$
  denotes projection onto the first factor) and the sequences
   \begin{equation*}
    \begin{tikzcd}
      0 \ar[r] & 0\ar[r] &NK^{-}_{n+1}(R) \ar[r, "{\partial}","\cong"'] &
      \Nil_n(R)\ar[r]& 0 \ar[r] & 0
    \end{tikzcd}
  \end{equation*}
  identifying via the boundary map $NK^{-}_{n+1}(R) \cong \Nil_n(R)$
  for $n\geq 0$ as desired.
\end{proof}

We may interpret these results in terms of $K$-theory spectra. Denote by $\K(R)$ the connective $K$-theory spectrum associated to the
space $K(R)$. We use $A\langle n \rangle$ to denote the $(n-1)$-connected
cover of the spectrum $A$, \ie
$\pi_*(A\langle n \rangle) \isom \pi_* A$ when $* \geq n$ and is $0$ for $*<n$.

\begin{theorem}[Fundamental Theorem for simplicial rings]\label{fundthm-simp-rings}
  For any simplicial ring $R$, there is a map of spectra
  \begin{equation*}
    \K(R) \to \susp^{-1} \hocofib \left( \K(R[t]) \vee_{\K(R)} \K(R[t\inv]) \to \K(R[t,t\inv]) \right)
  \end{equation*}
  which is functorial in $R$, induces equivalences on $(-1)$-connected
  covers (\ie $\pi_*$-isomorphism for $*\geq 0$), and splits a copy of
  $\K(R)\langle 0 \rangle$ off of
  $\susp\inv \K(R[t,t\inv]) \langle 0 \rangle$.
\end{theorem}
\begin{proof}[Proof of Theorem \ref{fundthm-simp-rings}]

  Let $l_+$ and $l_-$ denote the basechange maps
  $- \ox_{R[t]} R[t,t\inv]$ and $- \ox_{R[t\inv]} R[t,t\inv]$. Note
  that these induce localizing exact functors
  $l_+:\Proj(R[t]) \to \Proj(R[t,t\inv])$ and
  $l_-:\Proj(R[t\inv]) \to \Proj(R[t,t\inv])$. Combined with $i$ and
  $j$, they form the following square of exact functors:
  \begin{equation*}
    \begin{tikzcd}
      \cat H_1 (\bbP^1(R)) \ar[r, "j"] \ar[d, "i"] & \cat H_1(R[t\inv]) \ar[d, "l_-"] \\
      \cat H_1(R[t]) \ar[r, "l_+"] & \cat H_1(R[t,t\inv])
    \end{tikzcd}
  \end{equation*}
  The assignment $(M_+,M_-,\alpha) \mapsto \alpha$ defines a natural
  isomorphism $l_+\comp i \cong l_- \comp j$. Passing to $K$-theory
  spaces, we have a homotopy commuting square whose rows are
  localization sequences:
  \begin{equation*}
    \begin{tikzcd}[row sep=1em]
      K(\mathbf{Nil}(R)) \ar[d, "\simeq"] & K(\bbP^1(R))) \ar[d, "\simeq"] & K(R[t\inv]) \ar[d,"\simeq"] \\
      K(\cat H_{1,T}(R[t])) \ar[dd, dotted, "\id"] \ar[r] & K(\cat H_1 \bbP^1(R)) \ar[r, "j"]
      \ar[dd, "i"] &
      K(\cat H_1(R[t\inv])) \ar[dd, "l_-"] \\ \\
      K(\cat H_{1,T}(R[t])) \ar[r] & K(\cat H_1(R[t])) \ar[r, "l_+"] & K(\cat H_1(R[t,t\inv])) \\
      K(\mathbf{Nil}(R)) \ar[u, "\simeq"] & K(R[t]) \ar[u, "\simeq"] &
      K(R[t,t\inv]) \ar[u, "\simeq"]
    \end{tikzcd}
  \end{equation*}
  The identifications denoted above and below the right square follow
  from Theorem \ref{resolution_thm}. The homotopy fiber in the top row
  is identified with $K(\cat H_{1,T}(R[t]) \simeq K(\mathbf{\Nil}(R))$
  in the proof of Theorem \ref{thm1} - see Equation
  (\ref{KH1TequalsKNil}) in particular. The bottom row's homotopy
  fiber is also $K(\cat H_{1,T}(R[t])$ by Lemma
  \ref{lemma-composite-nil}. As $i$ projects onto the $R[t]$-module,
  $i$ restricted to $\cat H_{1,T}(R[t])$ is the dotted identity map on
  the left. Therefore, the right square is homotopy Cartesian and we
  get the following Mayer--Vietoris-style long exact sequence on
  homotopy groups for $n\geq 0$.
  \begin{equation*}
    \begin{tikzcd}[column sep=small]
      \cdots \ar[r] & K_n (\bbP^1(R)) \ar[r] & K_n(R[t])
      \oplus K_n(R[t\inv]) \ar[r] & K_n(R[t,t\inv]) \ar[r, "\bdry"] &
      K_{n-1}(\bbP^1(R)) \ar[r] & \cdots
    \end{tikzcd}
  \end{equation*}

  Theorem \ref{projline} and Lemma \ref{lemma-composite-nil} identify
  $K_n(\bbP^1(R)) \cong K_n(R) \oplus K_n(R)$ by the isomorphism
  $(u_0, u_0-u_1)$. Recall that
  $K_n(R[t^\pm]) \isom K_n(R) \oplus NK_n^\pm(R)$. Then the first map
  in the long exact sequence is $i_*\oplus j_*$ which is the
  diagonal map $\Delta$ on the first copy of $K_n(R)$ and zero on the
  second. Hence the long exact sequence decomposes into short
  exact sequences for $n\geq 1$:
  \begin{equation*}
    \begin{tikzcd}
      0 \ar[r] & K_n(R) \ar[r, "\Delta"] & K_n(R[t]) \oplus
      K_n(R[t\inv]) \ar[r] & K_n(R[t,t\inv]) \ar[r] & K_{n-1}(R) \ar[r]
      & 0
    \end{tikzcd}
  \end{equation*}
  To get this sequence for $n=0$, we observe that for any simplicial
  ring $R$, the map $K_*(R) \to K_*(\pi_0(R))$ is an isomorphism for
  $*=0,1$ (and an epimorphism for $*=2$). The sequence at $n=0$
  follows from Bass's classical work \cite{hb1}.

  We can reinterpret this argument to say that on corresponding
  connective $K$-theory spectra we have a fiber sequence
  \begin{equation*}
    \begin{tikzcd}
      \K(R)\vee \K(R) \simeq \K(\bbP^1(R)) \ar[r] & \K(R[t]) \vee
      \K(R[t\inv]) \ar[r] & \K(R[t,t\inv])
    \end{tikzcd}
  \end{equation*}
  where the first copy of $\K(R)$ includes via the diagonal map into
  the middle terms and the second is nullhomotopic. Hence we have the
  diagram of (co)fiber sequences of connective spectra
  \begin{equation*}
    \begin{tikzcd}
      \K(R) \ar[r, equals] \ar[d] & \K(R)\ar[d, "\Delta"] \ar[r] & * \ar[d] \\
      \K(R)\vee \K(R) \ar[r] \ar[d] & \K(R[t])\vee \K(R[t\inv]) \ar[r] \ar[d] & \K(R[t,t\inv]) \ar[d] \\
      \K(R) \ar[r, "\simeq 0"] & \K(R[t]) \vee_{\K(R)} \K(R[t\inv]) \ar[r]
      & \K(R[t,t\inv])
    \end{tikzcd}
  \end{equation*}
  where the vertical maps form cofiber sequences. Since the first map
  in the lower sequence is nullhomotopic, moving the bottom sequence
  forward yields a map $p$ that induces a surjection on homotopy
  groups
\begin{equation*}
  \begin{tikzcd}
    \K(R[t]) \vee_{\K(R)} \K(R[t,t\inv]) \ar[r] & \K(R[t,t\inv])
    \ar[r, "p"] & \susp \K(R)
  \end{tikzcd}
\end{equation*}
hence the map from the homotopy cofiber is a weak equivalence on
$0$-connected covers:
\begin{equation*}
  \hocofib\left( \K(R[t])\vee_{\K(R)} \K(R[t,t\inv]) \to \K(R[t,t\inv]) \right)\langle 1 \rangle \simeq \susp \K(R) \langle 1 \rangle.
\end{equation*}
The splitting for this sequence follows from that for $\K(\Z)$ which is
a classical result of Loday \cite{loday}. From Loday's work, we know
that there is a map $s:\susp \K(\Z) \to \K(\Z[t,t\inv])$ from regarding
the unit $[t]\in K_1(\Z[t,t\inv])$ as a map $S^1\to \K(\Z[t,t\inv])$
and smashing with $\K(\Z)$. Hence we have the fiber sequence
\begin{equation*}
  \begin{tikzcd}
    \K(\Z) \ar[r, "\simeq 0"] & \K(\Z[t]) \vee_{K(\Z)} \K(\Z[t\inv])
    \ar[r] & \K(\Z[t,t\inv]) \ar[r, shift right] & \susp \K(\Z) \ar[l, shift right, "s"']
  \end{tikzcd}
\end{equation*}
which is split by $s$ after taking $0$-connected covers. We smash this
sequence with $K(R)$ and use the $K$-theoretic product to produce maps
to the sequence:
\begin{equation*}
  \begin{tikzcd}
    \K(R) \ar[r, "\simeq 0"] & \K(R[t]) \vee_{\K(R)} \K(R[t\inv]) \ar[r] &
    \K(R[t,t\inv]) \ar[r, shift right] & \susp \K(R) \ar[l, shift
    right, "s'"']
  \end{tikzcd}
\end{equation*}
We observe that the composite
\begin{equation*}
  \begin{tikzcd}
    \susp \K(R) \ar[r] & \K(R) \wedge \susp^\oo(S^1_+) \ar[r] & \K(R)
    \wedge \susp \K(\Z) \ar[r] & \susp \K(R)
  \end{tikzcd}
\end{equation*}
is the identity up to homotopy, and the square
\begin{equation*}
  \begin{tikzcd}
    \K(R) \wedge \K(\Z[t,t\inv]) \ar[r, shift right] \ar[d] & \K(R) \wedge \susp \K(\Z) \ar[l, shift right, "\id \wedge s"'] \ar[d] \\
    \K(R[t,t\inv]) \ar[r, shift right] & \susp \K(R) \ar[l, shift right, "s'"']
  \end{tikzcd}
\end{equation*}
commutes, so $s'$ splits the sequence for the $K$-theory of $R$ after
passing to $0$-connected covers. Desuspending this sequence and
splitting completes the proof of the theorem.
\end{proof}

\section{Extending the Fundamental Theorem to connective $\bbS$-algebras}\label{connrings}

Following \cite[IV.10]{cw2}, define functors from
$(\textit{$\bbS$-algebras})$ to $(\textit{spectra})_*$ by $F_0(A)
\coloneq K(A)$ and inductively define $F_n(A)$ to be the total
homotopy cofiber (\ie iterated cofiber) of the square
\begin{equation*}
  \begin{tikzcd}
    F_{n-1}(A) \ar[r] \ar[d] & F_{n-1}(A[t]) \ar[d] \\
    F_{n-1}(A[t\inv]) \ar[r] & F_{n-1}(A[t,t\inv])
  \end{tikzcd}
\end{equation*}
In terms of these functors, the Fundamental Theorem is equivalent to
\begin{theorem}\label{thm:fund} For a connective $\bbS$-algebra
  $A$, there is a map of spectra
  $K(A) = F_{0}(A)\to \susp^{-n} F_{n}(A)$ (for any $n\geq 0$),
  functorial in $A$, which induces an equivalence between
  $K(A) = F_{0}(A)$ and the connective cover
  $\susp^{-n} F_{n}(A)\langle 0 \rangle$ of $\susp^{-n} F_{n}(A)$.
\end{theorem}
The canonical element
$[t_1,\ldots, t_n] \in K_1( \bbS[t_1^{\pm 1}, t_2^{\pm 1}, \ldots,
t_n^{\pm 1}])$ represented by the unit $(t_1,t_2,\ldots, t_n)$ induces
a map
\begin{equation*}
  S^n \wedge K(A) = S^n \wedge F_0(A) \to K(\bbS[t_1^{\pm 1},\ldots,
  t_n^{\pm 1}]) \wedge F_0(A) \to K(A[t_1^{\pm 1}, \ldots, t_n^{\pm 1}]) \to F_n(A)
\end{equation*}
whose adjoint provides the transformation $F_0(-) = K(-) \to
\susp^{-1}F_n(-)$ in the theorem.

In just the case $n=1$, this corresponds to the standard description
of the Fundamental Theorem.

We prove the theorem using two lemmas and a density argument inspired
by \cite{dgm}.

\begin{lemma}\label{lem1} The theorem is true for simplicial rings
  $A$.
\end{lemma}

\begin{proof}
  In section \ref{simprings}, we proved the case $n=1$ in the form of
  the classical short exact sequence version of the Fundamental
  Theorem.

  Inductively, we assume the map $F_0(-)=K(-) \to \susp^{-(n-1)}
  F_{n-1}(-)$ is an equivalence on connective covers. Observe
  that $F_n(A)$ is the total homotopy cofiber of a $2n$-cube whose
  vertices are of the form $K(A[x_1, \ldots x_n])$ where each $x_i$
  ranges over $\{1, t_i, t_i\inv, t_i^{\pm 1}\}$ and the maps in the
  cube are analogous to those for the Fundamental Theorem. Isolate the
  $x_n$ index of the cube to produce four $2(n-1)$-cubes whose
  homotopy cofibers compute the following square:
  \begin{equation*}
    \begin{tikzcd}
      F_{n-1}(A) \ar[r] \ar[d] & F_{n-1}(A[t_n]) \ar[d] \\
      F_{n-1}(A[t_n\inv]) \ar[r] & F_{n-1}(A[t_n^{\pm 1}])
    \end{tikzcd}
  \end{equation*}
  Inductively, we have maps from the following square which are
  objectwise equivalences on $(-1)$-connected covers after taking
  $(n-1)$-desuspensions.
  \begin{equation*}
    \begin{tikzcd}
      K(A) \ar[r] \ar[d] & K(A[t_n]) \ar[d] \\ K(A[t_n\inv]) \ar[r] &
      K(A[t_n^{\pm 1}])
    \end{tikzcd}
  \end{equation*}
  Passing to total homotopy cofibers, this square computes $F_1(A)$
  and the previous computes $F_n(A)$ by construction. Hence, we have
  an equivalence $F_1(-)\simeq \susp^{-(n-1)} F_n(-)\langle 0
  \rangle$ which we desuspend and combine with the Fundamental Theorem
  to arrive at the equivalence $K(-) = F_0(-) \simeq \susp^{-1}
  F_1(-) \langle 0 \rangle \simeq \susp^{-n} F_n(-) \langle 0
  \rangle$.
\end{proof}

%% Ernie's writing
\begin{lemma}\label{lem2}
  If $S_1\to S_2 \to S_3$ and $T_1\to T_2 \to T_3$ are cofiber
  sequences of $(-1)$- and $(-2)$-connected spectra (respectively),
  and $\phi_i:S_i\to T_i$ are maps respecting these cofiber sequences
  \begin{equation*}
    \begin{tikzcd}
      S_1 \ar[r] \ar[d, "\phi_1"] & S_2 \ar[r]
      \ar[d, "\phi_2"] & S_3 \ar[d, "\phi_3"] \\
      T_1 \ar[r] & T_2 \ar[r] & T_3
    \end{tikzcd}
  \end{equation*}
  then if $\phi_1$ and $\phi_3$ are equivalences of $(-1)$-connected
  covers, then so is $\phi_2$.
\end{lemma}
\begin{proof}%[Proof of lemma \ref{lem2}]
  Take homotopy cofibers of the maps $\phi_i$ to produce a new cofiber
  sequence
  $$\hocofib(\phi_1) \to \hocofib(\phi_2) \to \hocofib(\phi_3).$$ By
  connectivity of the maps $\phi_1$ and $\phi_2$, $\hocofib(\phi_1)$
  and $\hocofib(\phi_2)$ have homotopy groups concentrated in degrees
  strictly below $(-1)$. Hence $\hocofib(\phi_2)$ does as well and the
  result follows.
\end{proof}

\begin{proof}[Proof of Theorem \ref{thm:fund}]
  % The canonical element $[t]\in K_1(S[t,t^{-1}])$ represented by the
  % unital element $t$ induces a map
  % \begin{equation*}
  %   S^1\wedge K(A) = S^1\wedge F_0(A)\to K(S[t,t^{-1}])\wedge F_0(A)\to K(A[t,t^{-1}])\to F_1(A)
  % \end{equation*}
  % whose adjoint provides the transformation $F_0(-) = K(-)\to
  % \Sigma^{-1} F_1(-)$.

  % moved above to right after theorem statement for clarity.

  Following \cite[3.1.10]{dgm}, we can resolve our $\bbS$-algebra $A$
  by an $n$-cube $(A)_S$ (where $S$ lies in $\cat P_n$, the poset of
  subsets of $\{1,2,\ldots, n\}$), $\cat P \to \bbS-algebras$ with three
  crucial properties:
  \begin{itemize}
  \item the $n$-cube is $id$-Cartesian,
  \item each vertex of the $n$-cube is the Eilenberg-MacLane spectrum
    of a simplicial ring except for $A_\varnothing = A$, and
  \item after puncturing $(A)_S$ by restricting to
    $S \neq \varnothing$, the remaining maps all arise from maps of
    simplicial rings save in one direction.
  \end{itemize}
  For notational convenience, we will assume that not-so-nice
  direction in the cube are maps $S' \to S' \cup \{n\}$ with
  $S' \in P_{n-1}$. We will also assume $n\geq 2$ for the following
  argument.

  We form two new Cartesian $n$-cubes and by applying $F_0(-)$ and
  $F_i(-)$ to the punctured cube $(A)_S |_{S \neq \varnothing}$
  and then completing the diagrams by forming homotopy
  limits. Specifically, define $X_S = F_0(A)_S$ and
  $X_\varnothing = \holim_{S\neq \varnothing} F_0(A)_S$ and
  likewise $Y_S = F_i(A)_S$ and
  $Y_\varnothing = \holim_{S\neq \varnothing} F_i(-).$ When $n=2$,
  we have the following two homotopy pullback cubes.
  \begin{equation*}
    \begin{tikzcd}
      X_\varnothing \ar[r] \ar[d] \pb & X_{\{1\}} = F_0
      (A)_{\{1\}}
      \ar[d] \\
      X_{\{2\}} = F_0 (A)_{\{2\}} \ar[r] & X_{\{1,2\}} = F_0
      (A)_{\{1,2\}}
    \end{tikzcd}
    \quad
    \begin{tikzcd}
      Y_\varnothing \ar[r] \ar[d] \pb & Y_{\{1\}} = F_i
      (A)_{\{1\}}
      \ar[d] \\
      Y_{\{2\}} = F_i (A)_{\{2\}} \ar[r] & Y_{\{1,2\}} = F_i
      (A)_{\{1,2\}}
    \end{tikzcd}
  \end{equation*}

  The aforementioned natural transformation
  $\susp^i F_0 \to F_i$ induces a map of cubes
  $\susp^i X_S\to Y_S$. Whenever $S\neq \varnothing$, the vertices are
  simplicial rings and $\susp P_S \to Q_S$ is an equivalence of
  $(-1)$-connected covers by Lemma \ref{lem1}.

  Write $P_\texttop$ for the subcategory $P_{n-1}$ of $P_n$ where
  $n \notin S$. Write $P_\textbottom$ for the subcategory of $P_n$
  with $n \in S$. Note that $\{n\}$ is the initial object in
  $P_\textbottom$. Since $X$ and $Y$ are both homotopy Cartesian, the maps
  \begin{equation*}
    \tohofib_{S\in P_{\texttop} - \varnothing} \susp^i X_S \to \tohofib_{S\in P_\textbottom - \{n\}} \susp^i X_S
  \end{equation*}
  and
  \begin{equation*}
    \tohofib_{S\in P_{\texttop} - \varnothing} Y_S \to \tohofib_{S\in P_\textbottom - \{n\}} Y_S
  \end{equation*}
  between total homotopy fibers are weak equivalences. We note that
  $\susp^i X_S$ and $Y_S$ factor through simplicial rings after
  restricting to $P_{\textbottom}$ or to
  $P_{\texttop}-\varnothing$. We conclude that
  $\tohofib_{S\in P_{\texttop}-\varnothing} \susp^i X_S$,
  $\tohofib_{S\in P_{\texttop} - \varnothing}$,
  $\holim_{S\in P_{\texttop}-\varnothing} \susp^i X_S$, and
  $\holim_{S \in P_{\texttop}-\varnothing} Y_S$ also lie in the image
  of simplicial rings. We are left with the following diagram of fiber
  sequences.
  \begin{equation*}
    \begin{tikzcd}
      \tohofib_{S \in P_{\texttop}-\varnothing} \susp^i X_S \ar[d]
      \ar[r] & \susp^i X_\varnothing \ar[r] \ar[d] & \holim_{ S \in
        P_{\texttop}-\varnothing} \susp^i X_S \ar[d] \\
      \tohofib_{S \in P_{\texttop}-\varnothing} Y_S \ar[r] &
      Y_\varnothing \ar[r] & \holim_{ S \in P_{\texttop}-\varnothing}
      Y_S
    \end{tikzcd}
  \end{equation*}
  The left and right vertical maps are equivalences on
  $(-1)$-connected covers by Lemma \ref{lem1} so the middle map is as
  well by Lemma \ref{lem2}.

  All that remains is to compare the result on the $n$-cubes to the
  desired result on $A$. \cite[Thm.~3.2.1]{dgm} shows that
  $K$-theory takes $\id$-Cartesian $n$-cubes to $(n+1)$-Cartesian
  $n$-cubes. Hence, $F_0(A) \to X_\varnothing$ and $F_i(A) \to
  Y_\varnothing$ are $(n+1)$-connected. As we take $n$ to infinity by
  including $P_n\subset P_{n+1}$, we observe that these become weak
  equivalences. This extends the desired result from the cubes
  constructed from simplicial rings to the $S$-module $A$.
\end{proof}

\section{Corollaries and conjectures}

\begin{remark}
  In \cite[\S 9]{bm3}, Blumberg and Mandell coin the term \term{Bass
    functor} for homotopy functors exhibiting the above type of
  behavior. In particular, they show that the topological Dennis trace
  $K(-)\to THH(-)$ is a transformation of Bass functors, at least for
  discrete rings. The above suggests that this particular result of
  theirs extends to the category of $\bbS$-algebras.
\end{remark}

A consequence of Theorem \ref{thm:fund} is that the usual machinery
associated with a spectral interpretation of the Fundamental Theorem
produces a natural \emph{non-connective} delooping of the $K$-theory
functor $A\mapsto K(A)$ on the category $\cat{CSA}$, via 
application of the natural transformations $K(-)\to \susp^{-n}F_n(-)$.
The result is a (potentially) non-connective functor
\begin{equation*}
  A\mapsto K^B(A) = \colim_n \susp^{-n}F_n(A)
\end{equation*}
differing from the deloopings arising from the ``plus'' construction
\cite{ekmm}, or iterations of Waldhausen's $wS_\bullet$-construction
\cite{fw2}, which are always connective. %By construction, this
%non-connective $K$-theory functor agrees with Lueck's for simplicial
%rings.
\begin{conjecture}
  We conjecture that the non-connective $K$-theory functor $K^B$
  agrees with the non-connective $K$-theory functor of \cite{BGT}.
\end{conjecture}

We can use a similar argument to show that, at least for connective
$\bbS$-algebras, the negative $K$-groups arising from $K^B$ 
%iteration of the above construction \cite[Cor.~IV.10.3]{cw2}
depend only on $\pi_0(A)$. This result appears as
\cite[Thm.~9.53]{BGT} for their negative $K$-groups but our proof is
independent of that result and more direct if the previous conjecture
holds.

\begin{theorem} For any connective $\bbS$-algebra, the augmentation
  $A\surj \pi_0(A)$ induces an isomorphism
  \begin{equation*}
    \pi_nK^B(A) \isom \pi_n K^B(\pi_0(A)), \qquad n\le 1.
  \end{equation*}
\end{theorem}

\begin{proof} For simplicial rings $R$, the map $R \to \pi_0(R)$ is
  $1$-connected, so $K^B(R) \to K^B(\pi_0(R))$ is $2$-connected. We
  can extend this result to connective $\bbS$-algebras $A$ by
  resolving $K^B(A)$ and $K^B(\pi_0(A))$ by simplicial rings as in the
  proof of theorem \ref{thm:fund}. Let $X_S$ be the resolution
  $n$-cube for $K^B(A)$ completed to a Cartesian $n$-cube, and $Y_S$
  likewise for $K^B(\pi_0(A))$. When $n=2$, we arrive at the following
  diagram for $X_S$.
  \begin{equation*}
    \begin{tikzcd}
      K^B(A) = K^B(A)_\varnothing \ar[r, dotted] & X_\varnothing \ar[d] \ar[r] \pb
      & X_{\{1\}} = K^B(A)_{\{1\}}
      \ar[d] \\
      & X_{\{2\}} = K^B(A)_{\{2\}} \ar[r] & X_{\{1,2\}} = K^B(A)_{\{1,2\}}
    \end{tikzcd}
  \end{equation*}
  We know that $X_S$ and $Y_S$ are simplicial rings when
  $S\neq \varnothing$ so the maps $X_S \to Y_S$ are $2$-connected.
  Following the proof of Theorem \ref{thm:fund}, we extend the desired
  result to $X_\varnothing\to Y_\varnothing$ by analyzing the induced
  maps between the fiber sequences.
  \begin{equation*}
    \begin{tikzcd}
      \tohofib_{S \in P_{\texttop}-\varnothing} X_S \ar[d]
      \ar[r] & X_\varnothing \ar[r] \ar[d] & \holim_{ S \in
        P_{\texttop}-\varnothing} X_S \ar[d] \\
      \tohofib_{S \in P_{\texttop}-\varnothing} Y_S \ar[r] &
      Y_\varnothing \ar[r] & \holim_{ S \in P_{\texttop}-\varnothing}
      Y_S
    \end{tikzcd}
  \end{equation*}
  Here, the left and right maps are $\pi_n$-isomorphisms for $n\leq 1$
  and surjections on $n=2$ from the simplicial ring case. The long
  exact sequence in homotopy groups shows that the middle is a
  $\pi_n$-isomorphism for $n\leq 1$.

  Finally, $K$-theory carries $\id$-Cartesian $n$-cubes of
  $\bbS$-algebras to $(n+1)$-Cartesian cubes by \cite[Thm.~3.2.1]{dgm},
  so the comparison maps $K^B(A) \to X_\varnothing$ and
  $K^B(\pi_0(A))\to Y_\varnothing$ will be $(n+1)$-connected. Even just
  at $n=2$, this extends the result to $K^B(A)\to K^B(\pi_0(A))$ as
  desired.
\end{proof}

\begin{definition} The NK-spectrum of an $\bbS$-algebra $A$ is
  $NK(A) \coloneq \hofib(K^B(A[t])\to K^B(A))$.
\end{definition}

To make the notation correspond with convention, we should set
$NK^+(A) \coloneq NK(A)$ as just defined, and
$NK^-(A) \coloneq \hofib(K^B(A[t^{-1}])\to K^B(A))$. In this way, we
arrive at a more conventional formulation of Theorem \ref{thm:fund}:

\begin{theorem} For a connective $\bbS$-algebra $A$, there is a
  functorial splitting of spectra
  \begin{equation*}
    K^B(A[t,t^{-1}])\simeq K^B(A)\vee \Omega^{-1}(K^B(A))\vee NK^+(A)\vee NK^-(A)
  \end{equation*}
  where $\Omega^{-1}(K^B(A))$ denotes the non-connective delooping of
  $K^B(A)$ indicated above. Moreover, the involution $t\mapsto t^{-1}$
  induces an involution on $K^B(A[t,t^{-1}])$ which acts as the
  identity on the first two factors and switches the second two
  factors.
\end{theorem}
In the particular case $A = \susp^{\infty}(\Omega(X)_+)$ for a
connected pointed space $X$, we recover the main results of
\cite{hkvww1, hkvww2}.

Given the difficulty of computing $NK_*(R)$ for discrete rings, it is
not surprising that not much is known about $NK(A)$ for general
$\bbS$-algebras $A$. In the discrete setting, it is a classical result
of Quillen that $R$ Noetherian regular implies $NK(R)\simeq *$. This
fact led to the notion of \term{$NK$-regularity}; rings whose
$NK$-spectrum was contractible. Via the above discussion, the same
notion of $NK$-regularity may be extended to arbitrary
$\bbS$-algebras.

It has been shown by Klein and Williams \cite{kw1} that the map of
Waldhausen spaces arising from the Fundamental Theorem of
\cite{hkvww1} (and temporarily writing $A(X)$ for the Waldhausen
$K$-theory of the space $X$)
\begin{equation*}
  A(*)\vee \Omega^{-1} A(*)\to A(S^1)
\end{equation*}
is the inclusion of a summand but not an equivalence. In the notation
used here, $A(*) = K(\bbS)$ and $A(S^1) = K(\bbS[t,t^{-1}])$, where
$\bbS$ denotes the sphere spectrum. Thus (unlike the case of the
discrete ring $\Z$), one has
\begin{corollary} The sphere spectrum $\bbS$ is not $NK$-regular.
\end{corollary}
This result is not new to this paper and it follows additionally from
the computations of \cite{MR2399133} and \cite{MR2597738}. The nil
terms were further studied in \cite{MR2407063}.

\section{Technical proofs for the Fundamental
  Theorem}\label{technicalproofs}

 % Waldhausen structure on projective line
\begin{proposition}\label{proj_line_waldhausen_str}
  The category of projective modules over the projective line on $R$,
  $\Proj(\bbP^1(R))$, and the category of strict projective modules
  over the projective line, $\Proj^{str}(\bbP^1(R))$, admit the
  structure of Waldhausen categories where:
  \begin{itemize}
  \item $(0,0,\id)$ is the zero object,
  \item cofibrations are maps $(M_+,M_-,\alpha)\to (N_+,N_-,\beta)$
    that are cofibrations (\ie monomorphisms) $M_+\to N_+$ and $M_-\to
    N_-$, and
  \item weak equivalences are maps $f:(M_+,M_-,\alpha) \to (N_+,N_-\beta)$
  that are weak equivalences of simplicial modules $f_+:M_+\to N_+$
  and $f_- :M_-\to N_-$ (all maps are required to respect
  the structure isomorphisms $\alpha$ and $\beta$).
  \end{itemize}
  Moreover the inclusion $\Proj^{str}(\bbP^1(R)) \to \Proj(\bbP^1(R))$ induces a
  weak equivalence on $K$-theory spaces.
\end{proposition}
\begin{proof}
  All objects in $\Proj(\bbP^1(R))$ are clearly cofibrant, isomorphisms
  are also cofibrations, and pushouts of cofibrations are constructed
  coordinate-wise:
  \begin{equation*}
    \begin{tikzcd}
      (M_+,M_-, \alpha) \ar[r, tail] \ar[d] \po & (N_+, N_-,\beta) \ar[d] \\
      (O_+,O_-,\gamma) \ar[r, tail] & (O_+ \cup_{M_+} N_+ , O_- \cup_{M_-}
      N_-, \gamma \cup_\alpha \beta)
    \end{tikzcd}
  \end{equation*}
  Let $T$ denotes the cylinder functor on $R[t]$ and
  $R[t\inv]$-modules. For any morphism
  $f:(M_+,M_-,\alpha) \to (N_+, N_-, \beta)$, the pair
  $(\alpha, \beta)$ associates to the objects $(Tf_+, N_-)$ an isomorphism
  $\beta': Tf_+ \ox_{R[t]} R[t,t\inv] \to N_- \ox_{R[t\inv]}
  R[t,t\inv]$. This defines a cylinder functor
  $f\mapsto (Tf_+, N_-, \beta')$ on the strict modules over the
  projective line as well. Lemma \ref{lem2approx} then implies that
  the inclusion induces a weak equivalence on $K$-theory spaces.
\end{proof}

% Waldhausen structure on projective height $\leq n$ categories.
\begin{proposition}\label{proj_heights_waldhausen_str}
  The categories $\cat H_n (R)$ and $\cat H_n (\bbP^1(R))$ of
  projective height $\leq n$ modules admit Waldhausen category structures for
  $0 \leq n \leq \oo$ as follows:
  \begin{itemize}
  \item Weak equivalences are maps which are weak equivalences in
    $\Mod(R)$ or $\Mod(\bbP^1(R))$, respectively.
  \item Cofibrations are those maps which are cofibrations in
    $\Mod(R)$ or $\Mod(\bbP^1(R))$ whose cofibers lie in the
    subcategory $\cat H_n(R)$ or $\cat H_n(\bbP^1(R))$.
  \end{itemize}
\end{proposition}
\begin{proof}
  We note that the zero object lies in $\cat H_0$ and all properties
  of the Waldhausen category structure follow from the parent
  categories, $\Mod(R)$ and $\Mod(\bbP^1(R))$.
\end{proof}

\begin{theorem}[Theorem \ref{resolution_thm}]
  For any simplicial ring $R$, the inclusions $\cat H_0(R) \subseteq
  \cat H_n(R) \subseteq \cat H_\oo(R)$ and $\cat H_0(\bbP^1(R))
  \subseteq \cat H_n(\bbP^1(R)) \subseteq \cat
  H_\oo(\bbP^1(R))$ induce equivalences on algebraic $K$-theory,
  \begin{equation*}
    K(\cat H_0(R)) \simeq K(\cat H_n(R)) \simeq K(\cat H_\oo(R))
  \end{equation*}
  and
  \begin{equation*}
    K(\cat H_0(\bbP^1(R))) \simeq K(\cat H_n(\bbP^1(R)))
    \simeq K(\cat H_\oo(\bbP^1(R))).
  \end{equation*}
\end{theorem}
\begin{proof}
  This follows from the first author's PhD thesis \cite{fontes}, which
  extends Waldhausen's sphere theorem to the $K$-theory of stable
  $\oo$-categories using the language of weight
  structures.\footnote{Weight structures were introduced in
    \cite{bondarko} and independently in \cite{pauksztello} as
    co-$t$-structures. They are closely related to $t$-structures of
    triangulated categories but model cellular structures on the
    homotopy category. Using Barwick's construction of algebraic
    $K$-theory for $\oo$-categories from \cite{barwick}, the first
    author's thesis \cite{fontes} proves that a bounded weight
    structure on the homotopy category produces an equivalence on
    $K$-theory between a category and the heart of that weight
    structure.}  In order to apply these results, it suffices to
  demonstrate that the homotopy categories have bounded weight
  structures with hearts equivalent to $\cat H_0(-)$. We emphasize
  that this hypothesis is checked on the homotopy category as a
  triangulated category.
  \vskip.2in
  % categories as Waldhausen categories to the associated
  % $\oo$-categories, utilizing Barwick's construction of algebraic
  % $K$-theory for $\oo$-categories.
  
  % $\Mod_R$ forms a stable category \cite[6.3,
  % Prop.~1]{quillen}, so $\Mod_{\bbP^1(R)}$ is stable as well.

  Let $\cat C$ denote either $\Ho\cat H_\oo(R)$ or
  $\Ho\cat H_\oo(\bbP^1(R))$. We define a weight structure on $\cat C$
  by letting $\cat C_{w \leq n}$ be the full subcategory on modules of
  projective dimension $\leq n$, and $\cat C_{w\geq n}$ the full
  subcategory on modules of projective height $\geq n$. Both are
  closed under retracts, direct sum, and isomorphism in the homotopy
  category by definition. The heart of the weight structure, the
  intersection of $\cat C_{w\leq 0}$ and $\cat C_{w\geq 0}$, is
  $\cat C_{\heart w}=\cat C_{w=0}$ which equivalent to $H_0(R)$ or
  $H_0(\bbP^1(R))$.
  
  Since suspension in $\cat C$ can be modeled by $-\ox S^1$ (tensoring
  with the simplicial set $S^1$) $\Sigma$ shifts weights up as
  expected. We must also check that the weight structure has the
  desired orthogonality condition on maps. Suppose
  $X\in \cat C_{w\leq n}$ and $Y \in \cat C_{w\geq n+1}$. $X$ is a
  retract of $R\ox \twiddle X$ (or of
  $(R[t\inv ]\ox \twiddle X_-, R[t]\ox \twiddle X_+)$) with
  $\twiddle X$ (respectively, $\twiddle X_-$ and $\twiddle X_+$) an
  $n$-skeletal simplicial set. We see that $\cat C(X, Y)=0$ by the
  cell structure on simplicial sets.

  Finally, we need to provide weight decompositions for objects in
  $\cat C$. If $\cat C = \cat H_\oo(R)$, some object $X$ is a retract
  of $R\ox \twiddle X$ for some simplicial set $\twiddle X$. Write
  $r:R\ox \twiddle X \to X$ for that retraction. Since
  \begin{equation*}
    \begin{tikzcd}
      \sk_n \twiddle X \ar[r] & \twiddle X \ar[r] & \twiddle X / \sk_n
      \twiddle X
    \end{tikzcd}
  \end{equation*}
  is a cofiber sequence in simplicial sets, we see that
  \begin{equation*}
    \begin{tikzcd}
      R\ox \sk_n \twiddle X \ar[r] & R \ox \twiddle X \ar[r] & R\ox \twiddle X/\sk_n \twiddle X
    \end{tikzcd}
  \end{equation*}
  will be a cofiber sequence in $\Mod_R$ where $R \ox \sk_n \twiddle
  X$ has projective height $\leq n$ and $R \ox \twiddle X / \sk_n
  \twiddle X$ has projective height $\geq n+1$. Let $A$ denote the
  image of $R\ox \sk_n \twiddle X$ in $X$. We get a cofiber sequence
  \begin{equation*} 
    \begin{tikzcd}
      A \ar[r] & X \ar[r] & {\cofiber(A \to X)}
    \end{tikzcd}
  \end{equation*}
  as a retract of the cofiber sequence for $R\ox \twiddle X$ % tk work
                                                             % on this
  above. Hence, $A$ is in $\cat C_{w\leq n}$ and the cofiber is in
  $\cat C_{w\geq n}$. If $\cat C = \cat H_\oo(\bbP^1(R))$, two cofiber
  sequences can be constructed similarly, one for $\twiddle X_-$ and
  one for $\twiddle X_+$, which give a weight decomposition for
  $(X_+,X_-, \alpha)$.

  This weight structure is evidently bounded on $\cat C$, so by
  \cite[Thm.~4.1]{fontes} we conclude that the inclusions
  $\cat H_0(R) \to \cat H_\oo(R)$ and
  $\cat H_0(\bbP^1(R)) \to \cat H_\oo(\bbP^1(R))$ induce equivalences
  on $K$-theory.
  
  Let $\cat C_n$ denote the stable closure of $\cat C_{w\leq n}$ in
  $\cat C$. That is, $\cat C_n$ is the stable closure of $\cat H_n(R)$
  or $\cat H_n(\bbP^1(R))$. The above weight structure restricts to
  one on $\cat C_n$ with the same heart. Hence
  $K(\cat C_{w=0}) \simeq K(\cat C_n)$ by \cite[Thm.~4.1]{fontes}. By
  the additivity theorem, suspension acts by $- \id$ on $K$-theory, so
  $K(\cat C_n)$ is equivalent to $K(\cat C_{w\leq n})$ as desired.
\end{proof}
\vskip.2in

\begin{theorem}[$K$-theory of the projective line, Theorem \ref{projline}]
  The maps $u_0$ and $u_1$ induce an equivalence
  \begin{equation*}
    K(R) \x K(R) \xrightarrow{\simeq}{} K(\cat H_0 (\bbP^1(R))) \simeq K(\bbP^1(R))
  \end{equation*}
  after a choice of sum on the loopspace (which is canonical after
  passing to homotopy groups).
\end{theorem}
The proof of this theorem requires several intermediary results with
the final proof at the end of the section.
\vskip.2in

Define $\Gamma : \bbP^1(R) \to \Mod_R$ taking $(M_+, M_-, \alpha)$ to
the $R$-module $\Gamma(\alpha)$, which is the homotopy pullback in
this square in $\Mod_R$
\begin{equation}\label{eqn:gamma1}
  \begin{tikzcd}
    \Gamma(\alpha) \ar[r] \ar[d] \pb & M_+ \ar[d, "\alpha"] \\
    M_- \ar[r] & M_- \ox_{R[t\inv]} R[t,t\inv]
  \end{tikzcd}
\end{equation}
where all modules forget their $t$-action. Equivalently,
$\Gamma(\alpha)$ is the homotopy fiber of the map
\begin{equation}\label{eqn:gamma2}
  \begin{tikzcd}
    M_+ \x M_- \ar[r, "{(-\alpha,\id)}"] & M_- \ox_{R[t\inv]}
    R[t,t\inv]
  \end{tikzcd}
\end{equation}
which again lies in $\Mod_R$.
\vskip.2in

\begin{lemma}\label{Gamma:lemma1}
  $\Gamma$ defines an exact functor
  \[
  \Gamma: \Proj(\bbP^1(R))\to \Proj(R)
  \]
\end{lemma}
By construction, $\Gamma$ is exact. The main task will be to show
that the image lies in $\Proj(R)$. 
\vskip.2in

We first observe that verification of the lemma reduces to the case the object $(M_+, M_-,\alpha)$ lies in $\Proj^{str}(\bbP^1(R))$; i.e., that both $M_+$ and $M_-$ are finitely generated projective, not just so up to weak equivalence. For if $(N'_+,N'_-,\beta')$ is a retract of $(N_+,N_-,\beta)$, then $\Gamma(\beta')$ will be a retract of $\Gamma(\beta)$; moreover, as $\Gamma$ preserves weak equivalences, $(N_+,N_-,\beta)\simeq (M_+,M_-,\alpha)$ implies $\Gamma(\beta)\simeq \Gamma(\alpha)$.
\vskip.2in

Suppose now that $M_+$ is a finitely-generated projective
$R[t]$-module. Then for some $N \geq 0$ there exist maps
\begin{equation*}
  \begin{tikzcd}
    M_+ \ar[r, "\iota", shift left, tail] & \displaystyle \bigoplus^N
    R[t] \ar[l, "\pi", shift left, two heads] 
  \end{tikzcd}
\end{equation*}
presenting $M_+$. But there exists a canonical isomorphism $R[t]\cong \bigoplus_{i\geq 0} t^iR$ of
$R$-modules. So if we write $M_+[m,n]\coloneq t^m M_+ / t^n M_+$ and
$M_+[p] \coloneq M_+[p,p+1]$, we can similarly decompose $M_+$ as an $R$-module.
\begin{proposition}
  For any finitely-generated projective $R[t]$-module $M_+$, there are
  isomorphisms of $R$-modules,
  \begin{equation*}
    M_+[m,n] \cong \bigoplus_{m\leq p < n} M_+[p] \qquad \textrm{and} \qquad
    M_+ \cong \bigoplus_{p \geq 0} M_+[p]
  \end{equation*}
  for $n > m \geq 0$.
\end{proposition}

By replacing $t$ with $t\inv$ we can produce an analogous decomposition for any
finitely-generated projective $R[t\inv]$-module $M_-$ as well. We
negatively grade these summands to agree with the power of $t$,
writing $R[t\inv] \cong \bigoplus_{i \leq 0} t^i R$.
\begin{proposition}
  For any finitely-generated projective $R[t\inv]$-module $M_-$, there
  are isomorphisms of $R$-modules,
  \begin{equation*}
    M_-[m,n] \cong \bigoplus_{n > p \geq m} M_-[p]
  \end{equation*}
  and
  \begin{equation*}
    M_- \cong \bigoplus_{p \leq 0} M_-[p]
  \end{equation*}
  for $n < m \leq 0$.
\end{proposition}
In fact, this decomposition is as a direct sum of isomorphic
finitely-generated projective $R$-modules.
\begin{proposition}
  If $M$ is either a finitely-generated projective $R[t]$-module
  ($M=M_+$) or a finitely-generated projective $R[t\inv]$-module
  ($M=M_-$), then for each $p$, $M[p]$ is a finitely-generated
  projective $R$-module.
\end{proposition}
\begin{proof}
  This follows from the diagram
  \begin{equation*}
    \begin{tikzcd}
      M[p] \ar[r, "{\iota [p]}", shift left, tail] & \ar[l, "{\pi[p]}",
      shift left, two heads] \left( \bigoplus^N R[t^{\pm 1}] \right) [p]
    \cong \bigoplus^N R
    \end{tikzcd}
  \end{equation*}
  where $\iota[p]$ and $\pi[p]$ denote the restrictions to the $p$-th
  summand.
\end{proof}

\begin{proposition}
  For $M_+$ a finitely-generated projective $R[t]$-module,
  multiplication by $t^i$ induces an isomorphism
  $M_+[p]\cong M_+[p+i]$.

  Likewise, multiplication by $t^{-i}$ induces an isomorphism
  $M_-[p] \cong M_-[p-i]$ for any finitely-generated projective
  $R[t\inv]$-module $M_-$.
\end{proposition}
\begin{proof}
  Multiplication by $t^i$ maps $M_+[p]$ surjectively to $M_+[p+i]$ by
  construction. Since $\iota[p]$ lands in a free module and commutes
  with the $t$-action, multiplication by $t^i$ must be injective
  because it is on a free module. The proof is identical for $M_-$.
\end{proof}

Since multiplication by $t$ (and $t\inv$) acts by shifting on these
decompositions, we conclude that the localizations
$M_+ \ox_{R[t]} R[t, t\inv]$ and $M_- \ox_{R[t\inv]} R[t,t\inv]$
decompose as well. As $R$-modules, we then have
\[
M_+\ox_{R[t]} R[t,t\inv] \cong \bigoplus_{n\in \Z} M(+)
\] 
where
$M(+) = M_+[0] \cong M_+[p]$. We will denote $\bigoplus_{n\in \Z} M(+)$ by $M(+)(-\oo,\oo)$. Likewise, set
\begin{equation*}
  M(-)(-\oo,\oo) \coloneq M_- \ox_{R[t\inv]} R[t,t\inv] \cong \bigoplus_{n\in \Z} M_-[n]\cong  \bigoplus_{n\in \Z} M(-)
\end{equation*}
Finally, we will write $M(+)[a,b]$ and $M(-)[a,b]$ for the finite sums on
corresponding indices. This notation is designed so
$M_+\cong \bigoplus_{p\geq 0} M_+[p]$ includes into
$M(+)(-\oo,\oo) \cong M_+ \ox_{R[t]} R[t,t\inv]$ as $M(+)[0,\oo)$, with a similar statement holding for $M_-$.
\vskip.2in

We now return to the proof that $\Gamma(\alpha)\in obj(\Proj(R))$. Combining eq.~(\ref{eqn:gamma2}) with the above, we see that $\Gamma(\alpha) = \Sigma^{-1}Cone(id-\alpha)$, where $(id-\alpha)$ is the middle vertical map in the diagram
\begin{equation}\label{diagram:acyclic}
  \begin{tikzcd}
    M(-)(-\oo, 0] \ar[r] \ar[d, equals] & M(-)(-\oo, 0] \oplus M(+)[0,\oo) \ar[r] \ar[d, "{\id -\alpha}"] & M(+)[0,\oo) \ar[d, "\overline{\alpha}"]  \\
    M(-)(-\oo,0] \ar[r] & M(-)(-\oo, \oo) \ar[r] & M(-)[1,\oo)
  \end{tikzcd}
\end{equation}
The horizontal rows are short-exact sequences of $R$-modules (the top is just the cofiber
sequence $M_- \cof M_-\oplus M_+ \to M_+$ in our new notation). The vertical map of short-exact sequences then induces a short-exact sequence of mapping cones (where $*\simeq C := Cone(M(-)(-\infty,0]\xrightarrow[=]{Id} M(-)(-\infty,0])$)
\begin{equation}
 C\inj Cone(id-\alpha)\surj Cone(\overline{\alpha})
 \end{equation}
\vskip.2in

Now $\alpha:M(+)(-\oo,\oo) \to M(-)(-\oo, \oo)$ is a weak equivalence and an 
$R[t,t\inv]$-module map, which (by definition) admits a homotopy inverse $\alpha^{-1}:M(-)(-\oo,\oo) \to M(+)(-\oo, \oo)$ that is also an $R[t,t^{-1}]$-module map. Therefore both maps are determined on summand
components by their respective restrictions to $M(\pm)[0]$. Denote by $\alpha|_{[p,q]}$ resp.~$\alpha^{-1}|_{[p,q]}$ the restriction of $\alpha$ to $M(+)[p,q]$ resp.~of $\alpha^{-1}$ to $M(-)[p,q]$ (with a similar notation for open intervals).

As both $M(+)[0]$ and $M(-)[0]$ are finitely generated $R$-modules, hence compact, we may choose $N > 0$ such that for all $p$
\begin{gather*}
im\left(\alpha|_{[p]}\right)\subset M(-)[p-N,p+N]\\
im\left(\alpha^{-1}|_{[p]}\right)\subset M(+)[p-N,p+N]
\end{gather*}
It follows that in the composition sequence
\[
\overline{\alpha}_{[N+1,\infty)}: M(+)[N+1,\infty)\overset{inc}{\hookrightarrow} M(+)[0,\infty)\xrightarrow{-\overline{\alpha}} M(-)(-\infty,\infty)\xrightarrow{pr} M(-)[1,\infty)
\]
the right-most map $pr$ is injective on $im(-\overline{\alpha}\circ inc)$, implying that there is a commuting diagram
\begin{equation*}
  \begin{tikzcd}
    M(+)[N+1,\oo) \ar[r, tail] \ar[d, "\overline{\alpha}_{[N+1,\oo)}"] & M(+)(-\infty,\oo) \ar[d, "\alpha"] \\
    M(-)[1,\oo) \ar[r, tail] & M(-)(-\infty,\oo)
  \end{tikzcd}
\end{equation*}

Now the composite map
\[
\begin{tikzcd}[column sep = 1.7em]
M(+)[N+1,\infty)\ar[r,tail, "inc"l] & M(+)(-\infty,\infty)\ar[r, "\alpha"] & M(-)(-\infty,\infty)\ar[r, "\alpha^{-1}"] &
M(+)(-\infty,\infty)\ar[r,"proj"] & M(+)[N+1,\infty)
\end{tikzcd}
\]
is homotopic to the identity on $M(+)[N+1,\infty)$, via composition with the homotopy $\alpha^{-1}\circ\alpha\simeq Id$. Via the factorization by $\overline{\alpha}_{[N+1,\infty)}$ arising from the previous diagram, there results a weak equivalence
\[
M(-)[1,\infty)\xrightarrow{\simeq} M(+)[N+1,\infty)\oplus A
\]
where $A = Cone\left(\overline{\alpha}_{[N+1,\infty)}\right)$. Repeating this argument in the other direction produces another weak equivalence
\[
M(+)[N+1,\infty)\xrightarrow{\simeq} M(-)[2N+1,\infty)\oplus B
\]
where $B = Cone\left(M(-)[2N+1,\infty)\xrightarrow{\overline{\alpha^{-1}}|_{[2N+1,\infty)}} M(+)[N+1,\infty)\right)$. The composition then yields an $R$-module weak equivalence
\[
\gamma: M(-)[1,\infty)\xrightarrow{\simeq} M(-)[2N+1,\infty)\oplus B\oplus A
\]
The key point now is that, by construction, $\gamma$ maps $M(-)[2N+1,\infty)\subset M(-)[1,\infty)$ (via restriction) to the first summand on the right-hand side by a map homotopic to the identity on $M(-)[2N+1,\infty)$. Hence

\begin{corollary} The map $\gamma$ induces an $R$-module weak equivalence 
\[
\begin{split}
M(-)[1,2N] = Cone\left(M(-)[2N+1,\infty)\rightarrowtail M(-)[1,\infty)\right)\\
\xrightarrow{\simeq}
B\oplus A \simeq Cone\left(M(-)[2N+1,\infty)\rightarrowtail M(-)[2N+1,\infty)\oplus B\oplus A\right)
\end{split}
\]
Consequently $A$ is an object in $\Proj(R)$.
\end{corollary}

 Next consider the diagram
\begin{equation*}
  \begin{tikzcd}
    M(+)[N+1,\oo) \ar[r,tail] \ar[d, "\overline{\alpha}_{[N+1,\oo)}"] & M(+)[0,\oo) \ar[d] \ar[d, "\overline{\alpha}"] \ar[r, twoheadrightarrow] & M(+)[0,N] \ar[d] \\
    M(-)[1,\oo) \ar[d] \ar[r, equals] & M(-)[1,\oo) \ar[d] \ar[r] & 0 \ar[d]\\
    A \ar[r,tail] & \text{Cone}(\overline{\alpha}) \ar[r, twoheadrightarrow] & \Sigma M(+)[0,N]
  \end{tikzcd}
\end{equation*}
Here the rows and columns are all homotopy fiber as well as homotopy cofiber sequences. The
bottom row is a short-exact sequence of $R$-modules that splits, as $\Sigma M(+)[0,N]$ is projective. This yields a weak equivalence $\text{Cone}(\overline{\alpha})\simeq A\oplus \Sigma M(+)[0,N]$, implying $\text{Cone}(\overline{\alpha})$ is weakly equivalent to a retract of $B\oplus A\oplus\Sigma M(+)[0,N]\simeq M(-)[1,2N]\oplus\Sigma(+)[0,N]$. Hence it is an object of $\Proj(R)$, implying the same for its desuspension $\Sigma^{-1}\text{Cone}(\overline{\alpha})$.\qed
\vskip.2in

\begin{definition}
  For $i\geq 0$, define functors \term{$\ell_i$} on $\bbP^1(R)$ that
  take $(M_+, M_-, \alpha)$ to $(M_+, M_-, t^{-i}\comp \alpha)$. Write
  \term{$\Gamma_i$} for $\Gamma\comp \ell_i$, so $\Gamma =
  \Gamma_0$. Define functors $\term{u_i:} \Mod_R \to \cat H( \bbP^1(R))$ that
  take a (compact) module $M$ to $(M[t], M[t\inv], t^{i})$.
\end{definition}
Note that $\Gamma_0\comp u_0 \simeq \id$ by construction. 
\begin{definition}
  Let \term{$w_i$} denote the maps in $\bbP^1(R)$ that $\Gamma_i$
  takes to equivalences in $\Mod_R$.  Let \term{$\cat H_0
    (\bbP^1(R))^{w_i}$} denote the acyclics for $w_i$-localization,
  \ie the subcategory of modules which $\Gamma_i$ takes to a
  homotopically trivial $R$-module.
\end{definition}
We consider localizations of $\cat H_0 (\bbP^1(R))$ at these new
equivalences.

\begin{lemma}\label{acyclics-lemma}
  For all $i\geq 0$, $\Gamma_i$ and $u_i$ induce equivalences of
  Waldhausen categories between $\Proj(R)$ and
  $(\cat H_0 (\bbP^1(R))^{w_{i-1}}, w_i)$. (The $w_{-1}$-acyclics are
  defined to be the whole category $\cat H_0 (\bbP^1(R))$.)
\end{lemma}
\begin{proof} It's obvious by construction that $\Gamma_i \comp u_i \simeq
  \id$.  It's also straightforward to see that $u_i$ takes values in
  $w_{i-1}$-acyclics, since $\Gamma_{i-1}\comp u_i(M)$ fits into
   a homotopy pullback diagram
  \begin{equation*}
    \begin{tikzcd}
      \Gamma_{i-1} \comp u_i(M) \ar[r] \ar[d] \pb & M(-\oo, 0] \ar[d] \\
      M[0,\oo) \ar[r, "t^{-(i-1)} \comp t^i"] & M(-\oo,\oo)
    \end{tikzcd}
     \end{equation*}
 where the bottom composite is the map $t$ which shifts $M[0,\oo)$ isomorphically to
  $M[1,\oo)$.
\vskip.2in

  There is a natural transformation from $u_i\comp \Gamma_i$ to the
  identity which arises from the structure maps in the definition of
  $\Gamma_i$. Since $M(+)[0,\oo)$ is the $R[t]$-module $M_+$ with the
  $t$-action forgotten, the map $\Gamma_i(\alpha) \to M(+)[0, \oo)$
  corresponds to a map $\Gamma_i(\alpha)[t] \to M_+$ under the
  free--forget adjunction. Likewise, we have a map
  $\Gamma_i(\alpha)[t\inv]$ to $M_-$. This constructs the necessary
  natural transformation objectwise and the naturality is implied by
  the structure maps for $\Gamma_i$ commuting with
  $t^{-i}\comp\alpha$. This map is not an equivalence on
  $\bbP^1(R)$-modules, but is an equivalence after taking $\Gamma_i$,
  which is precisely what we need. % kt we ok: check i vs i+1 or i-1?
\end{proof}

\begin{proof}[Proof of Theorem \ref{projline}]
  Apply Waldhausen's fibration theorem \cite[Thm.~1.6.4]{fw2} to the
  localization of $\cat H_0 (\bbP^1(R))$ at $w_0 \cap w_1$ to produce
  the fiber sequence:
\begin{equation*}
  \begin{tikzcd}
    K(\cat H_0 (\bbP^1(R))^{w_0\cap w_1}) \ar[r] & K(\cat H_0 (\bbP^1(R))) \ar[r] &
    K(\cat H_0 (\bbP^1(R)), w_0\cap w_1)
  \end{tikzcd}
\end{equation*}
The acyclics for this localization are modules $(M_+, M_-, \alpha)$
with $\Gamma(\alpha)$ and $\Gamma_1(\alpha)$ trivial in
$\Mod_R$. 
\vskip.1in

\begin{claim} If $\Gamma_0(M_+, M_-, \alpha)\simeq \Gamma_1(M_+, M_-, \alpha)\simeq *$, then $(M_+, M_-, \alpha)\simeq *$.
\end{claim}

\begin{proof} Referring to diagram (\ref{diagram:acyclic}), we see that $\Gamma_0(M_+, M_-, \alpha)\simeq *$ implies $\overline{\alpha}:M(+)[0,\infty)\xrightarrow{\simeq} M(-)[1,\infty)$. But the weak equivalence $\overline{\alpha}$ is also represented by the composition
\begin{equation}\label{eqn:equiv}
M(+)[0,\infty)\xrightarrow[\cong]{t\cdot_-}M(+)[1,\infty) = t\cdot M(+)[0,\infty)\hookrightarrow M(+)[0,\infty)\xrightarrow{t^{-1}\overline{\alpha}} M(-)[1,\infty)
\end{equation}
The hypothesis $\Gamma_1(M_+, M_-, \alpha)\simeq *$ yields that $t^{-1}\overline{\alpha}$ is a weak equivalence. Hence the inclusion $M(+)[1,\infty)\hookrightarrow M(+)[0,\infty]$ in (\ref{eqn:equiv}) is also a weak equivalence, in turn implying
\[
M(+) = M(+)[0] = M(+)[0,\infty)/M(+)[1,\infty)\simeq *.
\]
But then $M_+ \cong M(+)[0,\infty)\simeq M(+)(-\infty,\infty)  \simeq *$, and so $*\simeq M(-)(-\infty,\infty)\simeq M(-)[-\infty,0]\cong M_-$, completing the proof of the claim.
\end{proof}
\vskip.2in
 We conclude that $K(\cat H_0 (\bbP^1(R))^{w_0 \cap w_1}) \simeq *$ and thus
$K(\cat H_0 (\bbP^1(R))) \simeq K(\cat H_0 (\bbP^1(R)), w_0 \cap w_1)$.
Localizing $(\cat H_0 (\bbP^1(R)), w_0 \cap w_1)$ at $w_0$ and use
Waldhausen's fibration theorem again to produces the fiber sequence
\begin{equation*}
  \begin{tikzcd}
    K( ( \cat H_0 (\bbP^1(R)))^{w_0}, w_1) \ar[r] & K(\cat H_0
    (\bbP^1(R)), w_0\cap w_1) \ar[r] & K(\cat H_0 \bbP^1(R), w_0)
  \end{tikzcd}
\end{equation*}

By Lemma \ref{acyclics-lemma}, the left term and the right term are each
equivalent to $K(R)$. The preceding argument shows that the middle
term is equivalent to $K(\cat H_0 (\bbP^1(R)))$. The identifications
with $K(R)$ are via $u_1$ (on the left term) and $u_0$ (on the
right term). Since $u_0$ splits $\Gamma_0$ and hence this
localization, the cofiber sequence splits to identify
$K(\cat H_0 (\bbP^1(R)))\simeq K(R)\x K(R)$.
\end{proof}

\bibliographystyle{amsplain}
\bibliography{refs_ftktsa}

\end{document}